\documentclass[11pt,leqno]{amsart}
\usepackage{amscd}
\usepackage{color}
\usepackage[symbol]{footmisc}
\usepackage{amssymb}
\usepackage{amsfonts}
\usepackage{latexsym}
\usepackage{verbatim}
\usepackage{braket}
\newcommand\norma[1]{\left\lVert#1\right\rVert}
\newcommand\abs[1]{\left\lvert#1\right\rvert}

\newcommand{\R}{\mathbb{R}}
\newcommand{\C}{\mathbb{C}}
\renewcommand{\H}{\mathbb{H}}
\renewcommand{\epsilon}{\varepsilon}
\renewcommand{\theta}{\vartheta}
\renewcommand{\phi}{\varphi}
\renewcommand{\Re}{\mathrm{Re}}
\renewcommand{\Im}{\mathrm{Im}}

\theoremstyle{plain}
\newtheorem{teor}{Theorem}
\newtheorem{prop}[teor]{Proposition}
\newtheorem{lem}[teor]{Lemma}

\theoremstyle{remark}

\theoremstyle{definition}

\title[]{A remark on the quaternionic Monge-Amp\`ere equation on foliated manifolds}

\begin{document}
 
\thanks{This work was supported by GNSAGA of INdAM}
\address{Dipartimento di Matematica G. Peano \\ Universit\`a di Torino\\
Via Carlo Alberto 10\\
10123 Torino\\ Italy}
\email{giovanni.gentili@unito.it} \email{luigi.vezzoni@unito.it}

\author{Giovanni Gentili and Luigi Vezzoni}
\date{\today}

\maketitle
\begin{abstract}
Pursuing the approach in \cite{GentiliVezzoni} we study the quaternionic Monge-Amp\`ere equation on HKT manifolds admitting an HKT foliation having corank $4$. We show that in this setting the quaternionic Monge-Amp\`ere equation has always a unique solution for every basic datum.  This approach includes the study of the equation on ${\rm SU}(3)$. 

\end{abstract}
\section{Introduction}
A hypercomplex manifold is a $4n$-dimensional smooth manifold $ M $ equipped with a triple of complex structures $(I,J,K)$ satisfying the quaternionic relation 
$$
IJ=-JI=K\,.
$$   
A Riemannian metric $g$ on a hypercomplex manifold  is called {\em hyperhermitian} if it is compatible with each $I,J,K$. A  hyperhermitian metric $g$ induces the {\em HKT form} 
$$
\Omega=\omega_{J}+i\omega_{K}\,,
$$
where $\omega_{J}$ and $\omega_K$ are the fundamental forms of $(g,J)$ and $(g,K)$, respectively. The form $\Omega$ belongs to $\Lambda^{2,0}_I$, is related to the metric $g$ by the formula 
\begin{equation}\label{g}
\Omega(Z,W)=2g(JZ, W) \,,\quad \mbox{ for every }Z,W\in \Gamma(T^{1,0}_IM)
\end{equation}
and satisfies 
\begin{eqnarray}
&& J\Omega=\bar	\Omega\,,\label{1}\\
&& \Omega(Z,J\bar Z)> 0\,,\quad \mbox{ for every }Z\in \Gamma(T^{1,0}_IM)\,. \label{2}
\end{eqnarray}
Viceversa every $\Omega \in \Lambda^{2,0}_{I}$ satisfying \eqref{1} and \eqref{2} determines a metric $g$ via \eqref{g}.

\medskip 
A hyperhermitian manifold $(M,I,J,K,g)$ is called HKT if its HKT form $\Omega$ satisfies
$$
\partial \Omega=0\,,
$$
where $\partial$ is with respect to $I$. HKT structures have been studied intensively in the last years for mathematical and physical reasons (see e.g. \cite{Alesker-Verbitsky (2006),Banos,GF,GLV,Grantcharov-Poon (2000),Howe-Papadopoulos (2000),Ivanov,Swann,Verbitsky (2002),Verbitsky (2007),Verbitsky (2009)} and the references therein).

\medskip 
This note focuses on the following conjecture introduced by Alesker and Verbisky in \cite{Alesker-Verbitsky (2010)}: 

\medskip 
\noindent {\bf Conjecture [Alesker, Verbitsky].} {\em Let $(M,I,J,K,g)$ be a compact {\rm HKT} manifold. For every $F\in C^{\infty}(M)$ there exists a unique $(\varphi,b)\in C^{\infty}(M)\times \R_+$ such that 
\begin{equation}
\label{qMAe}
(\Omega+\partial\partial_J\varphi)^n=b\,{\rm e}^F\Omega^n\,,\quad \Omega+\partial\partial_J\varphi>0\,,\quad \int_M\varphi {\rm Vol}_g= 0\,. 
\end{equation}
} 
Equation \eqref{qMAe} is called the {\em quaternionic Monge-Amp\`ere} equation and involves the operator $\partial_J=J^{-1}\bar\partial J$. Note that the form  $\Omega_\varphi:=\Omega+\partial\partial_J\varphi$ belongs to $\Lambda^{2,0}_I$ for every $\varphi \in C^{\infty}(M)$ and the condition $\Omega_\varphi>0$ means that $\Omega_\varphi$  satisfies \eqref{2}. Since $\Omega_\varphi$ satisfies \eqref{1} for every $\varphi \in C^{\infty}(M)$, the condition $\Omega_\varphi>0$ implies that $\Omega_\varphi$ determines a new HKT metric $g_\varphi$. Hence the conjecture of Alesker and Verbitsky is the natural counterpart of the classical Calabi conjecture in the realm of HKT geometry. Even if the conjecture is still open there are some partial results suggesting that it should be true: solutions to the equation are unique \cite{Alesker-Verbitsky (2010)}; solutions to the equation satisfy a $C^0$ a priori estimate \cite{Alesker-Shelukhin (2017),Alesker-Verbitsky (2010),Sroka}; the conjecture is true under some extra assumptions  \cite{Alesker (2013),BGV,DS,GentiliVezzoni}. 

\medskip 
The research of the present paper moves from our previous paper \cite{GentiliVezzoni} where we studied the Alesker-Verbitsky conjecture on some principal torus bundles over a torus. When the bundles considered in \cite{GentiliVezzoni} are regarded as $T^4$-bundles over a $T^4$, then the equation reduces to the classical Poisson equation on the base \cite[Remark in Section 2]{GentiliVezzoni}. Here we generalize the construction to foliated HKT manifolds, where the foliation replaces the role of the fiber. More precisely we consider the following setting:  

\medskip  
We say that a foliation $\mathcal F$ on an HKT manifold $(M,I,J,K,g)$ is an {\em HKT foliation} if 
$$
\mbox{$T_x\mathcal F$ is $(I_x,J_x,K_x)$-invariant for every $x$ in $M$},
$$
where $T\mathcal F$ denotes the vector bundle induced by $\mathcal F$. A function $f$ is called {\em basic} with respect to a foliation $\mathcal F$ if $X(f)=0$ for every $X\in\Gamma(\mathcal F)$, where $\Gamma(\mathcal F) $ is the space of smooth sections of $T\mathcal F$. We denote by $C^{k}_B(M)$ the space of real $C^k$ basic functions on $(M,\mathcal F)$. Our main result is the following:

\begin{teor}\label{main}
Let $(M,I,J,K,g)$ be a compact {\rm HKT} manifold and let $\mathcal F$ be an HKT foliation of real corank $4$ on $M$.  Then the quaternionic Monge-Amp\`ere equation \eqref{qMAe} has a unique solution for every basic datum $F\in C^{\infty}_B(M)$. Moreover the solution is necessarily basic.    
\end{teor}  
  
\medskip 
\noindent {\bf Acknowledgements.} The authors are very grateful to
 Giulio Ciraolo and Luciano Mari for many useful conversations.   
  
\section{Proof of Theorem \ref{main}}
\label{section}
It is well-known that solutions to the quaternionic Monge-Amp\`ere equation \eqref{qMAe} on a compact HKT manifold are in general unique. This can, for instance, be observed as follows: \\
let $(\phi_1,b_1),(\phi_2,b_2)$ be two solutions to \eqref{qMAe} with $ b_1\geq b_2 $. Setting $ \Omega_i=\Omega+\partial \partial_J \phi_i $ we have that
\[
\partial \partial_J(\phi_1-\phi_2) \wedge \sum_{k=0}^{n-1} \Omega_1^k\wedge \Omega_2^{n-1-k}= \Omega_1^n-\Omega_2^n=(b_1-b_2)\mathrm{e}^F \Omega^n\geq 0\,.
\]
On the left hand-side we have a second order linear elliptic operator without free term applied to $ \phi_1-\phi_2 $ and from the maximum principle and the fact that $ \phi_1,\phi_2 $ have zero mean it follows $\phi_1=\phi_2 $. Hence we have also $ b_1=b_2 $ and the uniqueness follows. 

\medskip 
Now we consider the framework of Theorem \ref{main}: let  $(M,I,J,K,g)$ be a compact {\rm HKT} manifold equipped with an HKT foliation $\mathcal F$ of real corank $4$ and consider the quaternionic Monge-Amp\`ere equation
\begin{equation}
\label{qMAeb}
(\Omega+\partial\partial_J\varphi)^n=b\,{\rm e}^F\Omega^n\,,\quad \Omega+\partial\partial_J\varphi>0\,,\quad \int_M\varphi {\rm Vol}_g= 0\,,\quad F\in C^{\infty}_B(M).
\end{equation}
We have the following 

\begin{lem}\label{lemma3}
Let $\varphi \in C^{2}_B(M)$. Then
$$
\frac{(\Omega+\partial \partial_J\varphi)^n}{\Omega^n}=
\Delta\varphi
%+L(\nabla \varphi)
+Q(\nabla \varphi,\nabla \varphi)+1\,,
$$
where $\Delta$ is the Riemannian Laplacian of $g$
%, $L\in \Gamma(T^*M)$
and $Q\in \Gamma(T^*M\otimes T^*M)$ is negative semi-definite. 
\end{lem}
\begin{proof}
Since $\mathcal F_x$ is $I_x$-invariant for every $x\in M$, then $T\mathcal F\otimes \C$ splits 
as $T\mathcal F\otimes \C=T^{1,0}\mathcal F\otimes T^{0,1}\mathcal F$. 
Let $\{Z_1,\dots,Z_{2n}\}$ be a local $g$-unitary frame with respect to $I$ such that 
$$
\langle Z_3,\dots,Z_{2n}\rangle=\Gamma(T^{1,0}\mathcal F)\,.
$$
Let us denote the conjugate $ \bar{Z}_r $ by $ Z_{\bar r} $ for every $ r=1,\dots,2n $ and suppose
$$
J(Z_{2k-1})=Z_{\overline{2k}}\,,\quad \text{for every $ k=1,\dots,n $}. 
$$
These assumptions imply that the HKT form of $g$ takes its standard expression  
\[
\Omega=Z^{12}+Z^{34}+\dots +Z^{(2n-1)(2n)}
\]
where $\{Z^1,\dots,Z^{2n}\}$ is the dual coframe to $\{Z_1,\dots,Z_{2n}\}$ and by $Z^{ij}$ we mean $Z^i\wedge Z^j$.  

We can write 
\begin{align*}
[Z_r,Z_s]=\sum_{k=1}^{2n} B^k_{rs}Z_k\,, && [Z_r,\bar{Z}_s]=\sum_{k=1}^{2n} \left(B^k_{r\bar{s}}Z_k+B^{\bar{k}}_{r\bar{s}}{Z}_{\bar k}\right)\,,
\end{align*}
for some functions $\{B_{rs}^k,B^k_{r\bar{s}},B^{\bar k}_{r\bar{s}}\}$.

For a basic function $\varphi$ we have 
\[
\partial_J\phi=-J \bar \partial \varphi=-J\left({Z}_{\bar 1}(\varphi) Z^{\bar 1}+{Z}_{\bar 2}(\varphi) Z^{\bar 2}\right)={Z}_{\bar 1}(\phi) Z^2-Z_{\bar 2}(\phi) Z^1;
\]
and
\[
\begin{aligned}
\partial \partial_J \varphi=&\sum_{k=1}^{2n}(Z_kZ_{\bar 1}(\varphi) Z^{k2}-Z_kZ_{\bar 2}(\varphi) Z^{k1})
+\sum_{r<s}(-Z_{\bar 1}(\phi) B^{2}_{rs}Z^{rs}
+Z_{\bar 2}(\phi)B^{1}_{rs} Z^{rs})\\
=&\sum_{k=1}^{2n}(Z_kZ_{\bar 1}(\varphi) Z^{k2}-Z_kZ_{\bar 2}(\varphi) Z^{k1})+
\sum_{r<s}(Z_{\bar 2}(\phi)B^{1}_{rs}-Z_{\bar 1}(\phi) B^{2}_{rs}) Z^{rs}\\
=&(Z_1Z_{\bar 1}(\phi)+Z_2Z_{\bar 2}(\phi)) Z^{12}+\sum_{k=3}^{2n}(Z_kZ_{\bar 1}(\phi) Z^{k2}
-Z_kZ_{\bar 2}(\phi) Z^{k1})+\sum_{r<s}(Z_{\bar 2}(\phi)B^{1}_{rs}-Z_{\bar 1}(\phi) B^{2}_{rs}) Z^{rs}\,,
\end{aligned}
\]
Since $ \mathcal{F} $ is a foliation, $ B^1_{rs}=0=B^2_{rs} $ for $ 2<r<s $, thus
\[
\begin{aligned}
\partial \partial_J \varphi=&(Z_1Z_{\bar 1}(\phi)+Z_2Z_{\bar 2}(\phi)) Z^{12}+
\sum_{k=3}^{2n}\sum_{l=1}^{2n}(B_{k\bar{1}}^lZ_l(\varphi) Z^{k2}+
B_{k\bar{1}}^{\bar l}Z_{\bar l}(\varphi) Z^{k2}-B_{k\bar{2}}^{l}Z_l(\varphi) Z^{k1}
-B_{k\bar{2}}^{\bar l}Z_{\bar l}(\varphi) Z^{k1})\\
&+(Z_{\bar 2}(\phi)B^{1}_{12}-Z_{\bar 1}(\phi) B^{2}_{12})Z^{12}+\sum_{s=3}^{2n}(Z_{\bar 2}(\phi)B^{1}_{1s}-Z_{\bar 1}(\phi) B^{2}_{1s}) Z^{1s}+\sum_{s=3}^{2n}(Z_{\bar 2}(\phi)B^{1}_{2s}-Z_{\bar 1}(\phi) B^{2}_{2s}) Z^{2s}\\
=&(Z_1Z_{\bar 1}(\phi)+Z_2Z_{\bar 2}(\phi)+B^{1}_{12}Z_{\bar 2}(\phi)-B^{2}_{12}Z_{\bar 1}(\phi) ) Z^{12}\\
&+\sum_{k=3}^{2n}
(B^{1}_{1k}Z_{\bar 2}(\phi)-B^{2}_{1k}Z_{\bar 1}(\phi) +B_{k\bar{2}}^{1}Z_1(\varphi) +B_{k\bar{2}}^{\bar 1}Z_{\bar 1}(\phi)+B_{k\bar{2}}^{2}Z_2(\varphi) +B_{k\bar{2}}^{\bar 2}Z_{\bar 2}(\phi)) Z^{1k}\\
&+\sum_{k=3}^{2n} (B^{1}_{2k}Z_{\bar 2}(\phi)-B^{2}_{2k}Z_{\bar 1}(\phi) -B_{k\bar{1}}^1Z_1(\varphi) -
B_{k\bar{1}}^{\bar 1}Z_{\bar 1}(\phi)-B_{k\bar{1}}^2Z_2(\varphi) -
B_{k\bar{1}}^{\bar 2}Z_{\bar 2}(\phi)) Z^{2k}\,.
\end{aligned}
\]
By setting
\[
\begin{aligned}
P_k(\nabla \phi)&=B^{1}_{1k}Z_{\bar 2}(\phi)-B^{2}_{1k}Z_{\bar 1}(\phi) +B_{k\bar{2}}^{1}Z_1(\varphi) +B_{k\bar{2}}^{\bar 1}Z_{\bar 1}(\phi)+B_{k\bar{2}}^{2}Z_2(\varphi) +B_{k\bar{2}}^{\bar 2}Z_{\bar 2}(\phi)\,,\\
Q_k(\nabla \phi)&=B^{1}_{2k}Z_{\bar 2}(\phi)-B^{2}_{2k}Z_{\bar 1}(\phi) -B_{k\bar{1}}^1Z_1(\varphi) -
B_{k\bar{1}}^{\bar 1}Z_{\bar 1}(\phi)-B_{k\bar{1}}^2Z_2(\varphi)-
B_{k\bar{1}}^{\bar 2}Z_{\bar 2}(\phi)\,,
\end{aligned}
\]
we obtain
\[
\begin{aligned}
\frac{(\Omega+\partial\partial_J\varphi)^n}{\Omega^n}=& 1+Z_1Z_{\bar 1}(\phi)+Z_2Z_{\bar 2}(\phi)+B^{1}_{12}Z_{\bar 2}(\phi)-B^{2}_{12}Z_{\bar 1}(\phi)\\
&+ \sum_{j=3}^n\left(P_{2j}(\nabla \phi)Q_{2j-1}(\nabla \phi)-P_{2j-1}(\nabla \phi)Q_{2j}(\nabla \phi)\right)\,.
\end{aligned}
\]
Since the Nijenhuis tensor of $ J $ vanishes
we have
\[
\begin{aligned}
0&=[Z_1,Z_{\bar 1}]+J[Z_1,JZ_{\bar 1}]+J[JZ_1,Z_{\bar 1}]-[JZ_1,JZ_{\bar 1}]\\
&=\sum_{k=1}^{2n}(B^k_{1\bar 1}Z_k+B_{1\bar 1}^{\bar k}Z_{\bar k}+B^k_{12}JZ_k+B^{\bar k}_{\bar 2 \bar 1}JZ_{\bar k}-B^k_{\bar 2 2}Z_k-B^{\bar k}_{\bar 2 2}Z_{\bar k})
\end{aligned}
\]
and thus
\[
\begin{cases}
B_{1\bar 1}^{2k-1}-B_{\bar 2 \bar 1}^{\overline{2k}}-B^{2k-1}_{\bar 2 2}=0\,,\\
B_{1\bar 1}^{2k}+B_{\bar 2 \bar 1}^{\overline{2k-1}}-B^{2k}_{\bar 2 2}=0\,,\\
\end{cases} \qquad k=1,\dots,n\,.
\]
Moreover, since $ B^{2k-1}_{1\bar 1},B^{2k}_{1\bar 1},B^{2k-1}_{2\bar 2},B^{2k}_{2\bar 2} $ are purely imaginary, for $ k=1 $ we deduce
\[
\begin{cases}
B_{\bar 2 \bar 1}^{\bar{2}}=B_{1\bar 1}^{1}+B^{1}_{2\bar 2}=-B_{\bar 11}^{\bar{1}}-B^{\bar{1}}_{\bar 22}=-B^{2}_{21}\,,\\
B_{\bar 2 \bar 1}^{\bar{1}}=-B_{1\bar 1}^{2}-B^{2}_{2\bar 2}=B_{\bar 11}^{\bar{2}}+B^{\bar{2}}_{\bar 22}=-B^{1}_{21}\,,
\end{cases}
\]
but then $ B^{2}_{21} $ and $ B^1_{21} $ are both real and purely imaginary, yielding
\[
\begin{cases}
B^{1}_{21}=0\,,\\
B^2_{21}=0\,,\\
B^1_{1\bar1}+B^1_{2\bar 2}=0\,,\\
B^2_{1\bar1}+B^2_{2\bar 2}=0\,.
\end{cases}
\]
Writing $ X_r=\Re(Z_r) $ and $ Y_r=\Im(Z_r) $ for $ r=1,2 $, we see that
\[
Z_rZ_{\bar r}(\phi)=(X_r+iY_r)(X_r-iY_r)(\phi)=X_rX_r(\phi)+i[Y_r,X_r](\phi)+Y_rY_r(\phi)
\]
and also
\[
0=\sum_{k=1}^{2n}(B^k_{1\bar 1}+B^k_{2\bar 2})Z_k(\phi)=([Z_1,Z_{\bar 1}]+[Z_2,Z_{\bar 2}])(\phi)=2i([Y_1,X_1]+[Y_2,X_2])(\phi)
\]
so that
\[
Z_1Z_{\bar 1}(\phi)+Z_2Z_{\bar 2}(\phi)=X_1X_1(\phi)+Y_1Y_1(\phi)+X_2X_2(\phi)+Y_2Y_2(\phi)=\Delta \phi\,.
\]
Furthermore, from the vanishing of the Nijenhuis tensor it easy to observe that
\[
Q_{2j-1}(\nabla \phi)=-\overline{P_{2j}(\nabla \phi)}\,,\qquad Q_{2j}(\nabla \phi)=\overline{P_{2j-1}(\nabla \phi)}\,.
\] 
Thus we finally obtain
$$
\frac{(\Omega+\partial\partial_J\varphi)^n}{\Omega^n}=1+\Delta \phi- \sum_{k=3}^{2n}|P_{k}(\nabla \phi)|^2\,.
$$
The claim then follows by setting 
%\[
%L(\nabla \phi)=B^{1}_{12}\bar{Z}_2(\phi)-B^{2}_{12}\varphi_{\bar 1}
%\]
%and
\[
Q(\nabla \phi,\nabla \phi)=- \sum_{k=3}^n|P_{k}(\nabla \phi)|^2\,. \qedhere
\]
\end{proof}
From the previous lemma it follows that under our assumptions for a basic function $F$ equation \eqref{qMAe} reduces to 
\begin{equation}\label{eqngen}
\Delta\varphi
%+L(\nabla \varphi)
+Q(\nabla \varphi,\nabla \varphi)+1=b\,{\rm e}^F\,,\quad \int_M\varphi {\rm Vol}_g = 0\,. 
\end{equation}
We then focus on this last equation in the general setting of a compact Riemannian manifold.

\medskip 
In order to prove existence of solutions to \eqref{eqngen} we need to show some a priori estimates. 
%We mention in passing that since we aim to study equation \eqref{eqngen} on a compact Riemannian manifold we cannot directly apply the known estimates for the quaternionic Monge-Amp\`{e}re equation (see e.g \cite{Alesker (2013),Alesker-Shelukhin (2017),DS,Sroka}) as they rely on the HKT condition. 
The $ C^0 $ bound is obtained  by using the Alexandrov-Bakelman-Pucci estimate as done by  B\l{}ocki in the case of the complex Monge-Amp\`{e}re equation \cite{Blocki}. The precise result we need is the following theorem  due to Sz\'{e}kelyhidi

\begin{teor}[Proposition 10 in \cite{Szekelyhidi}]
\label{teor_Szekelyhidi}
Let $ B_1(0)\subseteq \R^m $ denote the unit ball centered at the origin. Assume that $ u\in C^2(\R^m) $ satisfies $ u(0)+\epsilon \leq \min_{\partial B_1(0)} u(x). $ Then there exists a constant $ C_m $ depending only on $ m $ such that
\begin{equation*}
\epsilon^m\leq C_m\int_{\Gamma_\epsilon} \det(D^2u)\,,
\end{equation*}
where $ D^2u $ is the Hessian of $ u $ and
\[
\Gamma_\epsilon =\left\{ x\in B_1(0) \mid  u(y)\geq u(x)+\nabla u(x)\cdot (y-x),\, \forall y\in B_1(0),\, \abs{\nabla u(x)}<\frac{\epsilon}{2} \right\}.
\]
\end{teor}

This theorem allows us to reduce the $ C^0 $ estimate to an $ L^p $ estimate by using the following 
\begin{teor}[Weak Harnack Inequality, Theorem 8.18 in \cite{GT}]
\label{teor_Harnack}
Let $ R>0 $ and fix an integer $ m>2 $. Assume $ u \in C^2(\R^m) $ is non-negative on $ B_R(0) $ and such that $ \Delta u(x)\leq f(x) $ for some $ f\in C^0(\R^m) $ and all $ x\in B_R(0) $. Consider $ 1\leq p <m/(m-2) $, and $ q>m $. Then there exists a positive constant $ C=C(m,R,p,q) $ such that
\begin{equation*}
r^{-m/p}\|u\|_{L^p(B_{2r}(0))}\leq C\left( \inf_{x\in B_r(0)} u(x)+r^{2-2m/q}\|f\|_{L^{q/2}(B_{R}(0))} \right),
\end{equation*}
for any $ 0<r<R/4 $.
\end{teor}

\begin{lem}
\label{C0}
Let $(M,g)$ be a compact Riemannian manifold, $Q\in \Gamma(T^*M\otimes T^*M)$ be negative semi-definite and $F\in C^{0}(M)$. If $ (\phi,b)\in C^2(M)\times \R_+ $ satisfies 
$$
\Delta\varphi +Q(\nabla \varphi,\nabla \varphi)+1=b\,{\rm e}^F\,,\quad \int_M\varphi {\rm Vol}_g = 0\,,
$$
then there exists a positive constant $C$ depending only on $M$, $g$, $ \|Q\|_{C^0} $ and $F$ such that
\[
\norma{\phi}_{C^0}\leq C\,, \qquad b\leq C\,.
\]
\end{lem}
\begin{proof}
First of all we bound the constant $ b $. At a maximum point $ p $ of $ \phi $ we have
%$ \nabla\phi =0 $ and $ \Delta \phi \leq 0 $, therefore
$ b\, \mathrm{e}^{F(p)}-1\leq 0 $ and thus $ b\leq \mathrm{e}^{-F(p)}\leq \|\mathrm{e}^{-F}\|_{C^0} $ so that the constant $ b $ is bounded.

Let $ x_0\in M $ be a point where $ \phi $ achieves its minimum and consider a coordinate chart centered at $ x_0 $. Without loss of generality we may identify this chart with a ball $ B_1(0)\subseteq \R^m $ of unit radius, where $ m=\dim(M) $. Fix $ \epsilon>0 $ and define
\[
u(x)=\phi(x)-\max_M\phi+\epsilon \abs{x}^2\,.
\]
Applying Theorem \ref{teor_Szekelyhidi} to $ u $ we have
\begin{equation}
\label{eq_epsilon}
\epsilon^m\leq C_m\int_{\Gamma_\epsilon}\det(D^2u)\,.
\end{equation}
 We aim to prove that $ D^2u $ is bounded on $ \Gamma_\epsilon $. Differentiating $ u $ we see that
\[
\nabla u= \nabla \phi+2\epsilon x\,,\qquad D^2u=D^2\phi+2\epsilon I_m\,,
\]
where $ I_m $ is the identity matrix. As a consequence $ u $ satisfies the following equation
\[
\Delta u-2m\epsilon +Q(\nabla u-2\epsilon x,\nabla u-2\epsilon x)+1=b\,{\rm e}^F\,.
\]
Set, for instance, $ \epsilon=1 $. Now, since on $ \Gamma_{1} $ the Hessian $ D^2u $ is non-negative, % and thus $ u_{ii}\geq 0 $
by the arithmetic-geometric mean inequality we deduce
\[
\begin{aligned}
\det(D^2u(x))&\leq \left( \frac{\Delta u(x)}{m} \right)^m\leq \left \lvert Q(\nabla u(x)-2x,\nabla u(x-2x)+2m-1+b\, \mathrm{e}^{F(x)} \right \rvert^m\\
&\leq \left(\norma{Q}_{C^0} \abs{\nabla u(x)-2x}^2+2m+1+b \left \lvert {\rm e}^{F(x)} \right \rvert \right)^m \\
&\leq \left(\frac52 \norma{Q}_{C^0}+2m+1+b \left \lvert {\rm e}^{F(x)} \right \rvert \right)^m\leq C\,,
\end{aligned}
\]
for any $ x\in \Gamma_1 $, where $ C>0 $ is a uniform constant.

Now we  observe that
\[
u(x)\leq u(0)-\nabla u(x)\cdot (-x)\leq u(0)+\frac{1}{2}\,,\qquad \text{for every } x\in \Gamma_{1}\,,
\]
which implies 
\[
\phi(x)-\max_M \phi+\abs{x}^2\leq \phi(0)-\max_M \phi+\frac12=\min_M \phi- \max_M \phi +\frac12\,, \qquad \text{for every } x\in \Gamma_{1}\,,
\]
and therefore
\[
\max_M \phi-\min_M\phi \leq \max_M \phi - \phi(x)+ \frac12\,, \qquad \text{for every } x\in \Gamma_{1}\,.
\]
It follows that for every $ p\geq 1 $ we have
\begin{equation*}
\left(\max_M \phi-\min_M \phi\right) \left \lvert \Gamma_{1} \right \rvert ^{1/p}\leq \left \lVert \max_M \phi-\phi+\frac12 \right \rVert_{L^p(\Gamma_{1})} \leq \left \lVert \max_M \phi-\phi+\frac12 \right \rVert_{L^p(B_1(0))}\,.
\end{equation*}
Combining this with \eqref{eq_epsilon} and the fact that $ \int_M\phi=0 $, we have 
\begin{equation*}
\norma{\phi}_{C^0}\leq \max_M \phi - \min_M \phi\leq \abs{\Gamma_1}^{-1/p}\left \lVert \max_M \phi-\phi+\frac12 \right \rVert_{L^p(B_1(0))}\leq C \left \lVert \max_M \phi-\phi\right \rVert_{L^p(B_1(0))}\,,
\end{equation*}
for every $ p\geq 1 $. In conclusion we only need to prove an $ L^p $ estimate for $ \max_M\phi- \phi $ to obtain the desired estimate on $ \phi $. Since $ Q $ is negative semidefinite we see from the equation that
\[
\Delta \phi \geq b\, {\rm e}^F-1\geq C\,{\rm e}^{F}-1\,,
\]
where we used that $ b $ is uniformly bounded. This entails that $ \Delta(\max_M\phi -\phi) \leq 1-C\,{\rm e}^{F} $, and applying Theorem \ref{teor_Harnack} to $ \max_M \phi -\phi $ with, $ 1\leq p \leq m/(m-2) $, $ q=2m $, $ r=1/2 $ and $ R=2 $ we infer
\begin{equation*}
\left \lVert \max_M \phi -\phi \right \rVert_{L^{p}(B_1(0))}\leq C\left( \inf_{B_{1/2}(0)}\left(\max_M \phi -\phi\right)+\frac12\norma{1-C\,{\rm e}^{F}}_{L^m(B_2(0))} \right)\leq C\,,
\end{equation*}
as required.
\end{proof}

For higher order bounds we need to recall the following two results:

\begin{teor}[Theorem 3.1, Chapter 4 \cite{LU}]\label{stimaC1}
Let $ \Omega \subseteq \R^n $ be a bounded connected open subset. Consider a semilinear elliptic equation of the following type
\[
\Delta u+a(x,u,\nabla u)=0\,,
\]
where the function $ a(x,u,p) $ is measurable for $ x\in \bar \Omega $ and arbitrary $ u\in \R $, $ p\in \R^n $ and satisfies
\[
(1+|p|)\sum_{i=1}^n|p_i|+|a(x,u,p)|\leq \mu(|u|) (1+|p|)^m\,,
\]
for some $ m>1 $ and some non-decreasing continuous function $ \mu \colon [0,+\infty) \to \R $. Let $ u\in C^2(\Omega) $ be a solution of the given equation, then, for any connected open subset $ \Omega'\subset \Omega $ there exists a constant $ C>0 $ depending only on $ \| u \|_{C^0(\Omega)} $, $ \mu(\| u\|_{C^0(\Omega)} ) $, $ m $ and $ d(\Omega',\partial \Omega) $ such that
\[
\| u \|_{C^1(\Omega')} \leq C\,.
\]
\end{teor}

\begin{teor}[Theorem 6.1, Chapter 4 \cite{LU}]\label{stimaC1alpha}
Let $ \Omega \subseteq \R^n $ be a bounded connected open subset. Consider a semilinear elliptic equation of the following type
\[
\Delta u+a(x,u,\nabla u)=0\,,
\]
where the function $ a(x,u,p) $ is measurable for $ x\in \bar \Omega $ and arbitrary $ u\in \R $, $ p\in \R^n $ and satisfies
\[
\| a \|_{C^0(\Omega)}<\mu_1\,,
\]
for some constant $ \mu_1<\infty $. Let $ u\in C^2(\Omega) $ be a solution of the given equation such that
\[
\| \nabla u \|_{C^0(\Omega)}< C\,,
\]
then there exists $ \alpha\in (0,1) $ depending only on $ \| \nabla u \|_{C^0(\Omega)} $ and $ \mu_1 $ such that $ \nabla u\in C^{0,\alpha}(\Omega,\R^n) $. Moreover, for any connected open subset $ \Omega'\subset \Omega $ there exists a constant $ C>0 $ depending only on  $ \| \nabla u \|_{C^0(\Omega)} $, $ \mu_1 $ and $ d(\Omega',\partial \Omega) $ such that
\[
\| u \|_{C^{1,\alpha}(\Omega')} \leq C\,.
\]
\end{teor}
We can then establish the higher order a priori estimates for solutions to \eqref{eqngen}.
\begin{lem}\label{lapl}
Let $(M,g)$ be a compact Riemannian manifold, 
%$L\in \Gamma(T^*M)$,
$Q\in \Gamma(T^*M\otimes T^*M)$ and $F\in C^{0}(M)$. If $ (\phi,b)\in C^2(M)\times \R_+ $ satisfies 
$$
\Delta\varphi%+L(\nabla \varphi)
+Q(\nabla \varphi,\nabla \varphi)+1=b\,{\rm e}^F\,,
$$
then there exists a positive constant $C$ depending only on $M$, $g$, $ \|\phi\|_{C^0} $, $ b $,  
%$ \|L\|_{C^0} $,
$ \|Q\|_{C^0} $ and $F$ 
such that
\[
\|\Delta \phi\|_{C^0}\leq C\,.
\]
\end{lem}
\begin{proof}
As an application of Theorem \ref{stimaC1} with $ a=
Q+1-b\, \mathrm{e}^F $, $ m=2 $, and $ \mu \equiv 
\|Q\|_{C^0}+b\|\mathrm{e}^F\|_{C^0}+\sqrt{n}+1 $ we have that there exists a constant $ C>0 $ such that
\[
\|\nabla \phi\|_{C^0}\leq C\,.
\]
Then from the equation we have
\[
\|\Delta \phi\|_{C^0}\leq b\,{\rm e}^F+1%+\| L(\nabla \varphi)\|_{C^0}
+\|Q(\nabla \varphi,\nabla \varphi)\|_{C^0} \leq C\,,
\]
and the claim follows.
\end{proof}

\begin{lem}\label{C2alpha}
Let $(M,g)$ be a compact Riemannian manifold, %$L\in \Gamma(T^*M)$, 
$Q\in \Gamma(T^*M\otimes T^*M)$ and $ F\in C^{k,\beta}(M) $. If $ (\phi,b)\in C^{2}(M)\times \R_+ $ solves
$$
\Delta\varphi%+L(\nabla \varphi)
+Q(\nabla \varphi,\nabla \varphi)+1=b\,{\rm e}^F\,,
$$
then there exists $ \alpha\in (0,\beta) $ such that $\varphi$ in $C^{k+2,\alpha}(M)$ and there is a positive constant $ C>0 $ depending only on $M$, $ g $, $ b $, $ \|\nabla \phi\|_{C^0} $, %$ \|L\|_{C^0} $,
$ \|Q\|_{C^0} $ and $F$ 
such that
\[
\|\phi\|_{C^{k+2,\alpha}}\leq C\,.
\]
\end{lem}
\begin{proof}
Lemma \ref{lapl} implies $ \|\nabla \phi \|_{C^0}\leq C $ and  we can apply Theorem \ref{stimaC1alpha} choosing the constant $ \mu_1=%\|L\|_{C^0}+
\|Q\|_{C^0}+b\|\mathrm{e}^F\|_{C^0}+1 $ and deduce that there exist $ \alpha \in (0,1) $ and a constant $ C>0 $ such that
\[
\| \phi \|_{C^{1,\alpha}}\leq C\,.
\]
Then the equation implies the estimate $ \|\Delta \phi\|_{C^{0,\alpha}} \leq C $, which can be improved to a $ C^{2,\alpha} $ estimate for $ \phi $ using Schauder theory by assuming $ \alpha<\beta $. Then $ \phi \in C^{2,\alpha}(M) $ and by a bootstrapping argument the claim follows.
\end{proof}

Now, we prove that equation \eqref{eqngen} is always solvable.

\begin{prop}\label{prop}
Let $(M,g)$ be a compact Riemannian manifold, $ Q\in \Gamma(T^*M\otimes T^*M) $ be negative semi-definite and $ F\in C^{k,\beta}(M) $. Then equation
\[
\Delta\varphi%+L(\nabla \varphi)
+Q(\nabla \varphi,\nabla \varphi)+1=b\,{\rm e}^F\,,\quad \int_M\varphi {\rm Vol}_g = 0
\]
admits a solution $ (\phi,b) \in C^{k+2,\alpha}(M)\times \R_+ $ for $ \alpha\in (0,\beta) $. 
\end{prop}
\begin{proof}
Let $F\in C^{k,\beta}(M) $ and consider the set  
$$
S:=\left\{t'\in [0,1]\,\,:\,\,\eqref{start} \mbox{ has a solution } (\phi_t,b_t)\in C^{4,\alpha}(M)\times \R_+ \mbox{ for }t\in[0,t']\right\}\,,
$$
where 
\begin{equation}
\tag{$*_t$}\label{start}
\Delta\varphi_t%+L(\nabla \varphi)
+Q(\nabla \varphi_t,\nabla \varphi_t)=b_t\,{\rm e}^{tF}-1\,,\quad \int_M\varphi_t {\rm Vol}_g = 0
\end{equation}
and $\alpha\in (0,\beta)$ is fixed. 

$S$ in not empty since the pair $ (\phi_0,b_0)=(0,1) $ solves \eqref{start} for $ t=0 $ and, therefore, $ 0\in S $. In  order to prove the statement we need to show that $S$ is open and closed in $[0,1]$. 

\medskip
To show that $S$ is open we apply, as usual, the inverse function theorem between Banach spaces. Let $\hat t\in S$ and $ (\phi_{\hat t},b_{\hat t}) $ be solution of $ (*_{\hat t}) $, let $B_1$ and $B_2$ be the Banach spaces 
\begin{eqnarray*}
 B_1:=\left\{\psi \in C^{4,\alpha}(M)\,\,:\,\,\int \psi {\rm Vol}_g =0\right\}\,,\quad   B_2:= C^{2,\alpha}(M)
\end{eqnarray*}
and let $\Psi \colon B_1 \times \R_+ \to B_2$ be the operator
$$ 
\Psi(\psi,a):=\log\left(\frac{\Delta \psi +Q(\nabla \psi,\nabla \psi)+1}{a}\right)\,.
$$
The differential $ \Psi_{*|(\phi_{\hat t},b_{\hat t})}\colon B_1\times \R \to B_2 $ is
$$
\Psi_{*|(\phi_{\hat t},b_{\hat t})}(\eta,c)=\frac{\Delta \eta+2Q(\nabla \eta, \nabla \phi_{\hat t})}{b_{\hat{t}}\, {\rm e}^{\hat{t}F}}-\frac{c}{b_{\hat t}}\,.
$$
Since $ T\colon \eta \mapsto \Delta \eta+2Q(\nabla \eta, \nabla \phi_{\hat t}) $ is a second order linear elliptic operator without terms of degree zero, by  the maximum principle its kernel is the set of constant functions on $M$. Moreover, since $T$ has the same principal symbol of the Laplacian operator it has index zero. Denoting with $ T^* $ the formal adjoint of $ T $ we then have
\[
\dim \ker(T^*)=\dim {\rm coker} (T)=\dim \ker(T)-{\rm ind}(T)=1\,.
\]
Let $ \rho \in C^{2,\alpha}(M) $. Equation
\begin{equation}\label{DQ}
\Delta \eta+2Q(\nabla \eta, \nabla \phi_{\hat t})=c\, {\rm e}^{\hat t F}+\rho\, b_{\hat t}\, {\rm e}^{\hat t F}
\end{equation}
is solvable if and only if its right hand side is orthogonal to $ \ker(T^*) $, or equivalently to a generator of $ \ker(T^*) $. This can always be accomplished by a suitable choice of the constant $c$ and, therefore, there always exists a solution $ \eta \in C^{4,\alpha}(M) $ to \eqref{DQ}. Moreover the solution $\eta$ is unique in $ B_1 $ because $ \int_M \eta{\rm Vol}_g=0 $. The differential $ \Psi_{*|(\phi_{\hat t},b_{\hat t})} $ is then an isomorphism and it follows by the inverse function theorem that the operator $ \Psi $ is locally invertible around $ (\phi_{\hat t},b_{\hat t}) $, implying that there exists $\epsilon>0$ such that for $t\in[\hat t,\hat t+\epsilon)$ equation \eqref{start} can be solved.

\medskip  
Next we prove that $S$ is closed. Let $ \{t_j\}$ be a sequence in $S$ converging to some $ t\in [0,1] $ and consider the corresponding solutions $ (\phi_j,b_j)=(\phi_{t_j},b_{t_j}) $ to $(*_{t_j})$. In view of Lemma \ref{C0} the families $ \{\norma{\phi_{j} }_{C^0}\} $, $\{b_{j}\}$ are uniformly bounded from above. Moreover, Lemmas \ref{lapl} and \ref{C2alpha} imply that the family $ \{\phi_j\}$ is uniformly bounded in $ C^{k+2,\alpha} $-norm. Consequently, by Ascoli-Arzel\`a Theorem, up to a subsequence, $ \phi_j$ converges to some $\phi_t\in C^{k+2,\alpha}(M)$ in $ C^{k+2,\alpha}$-norm and $ b_j$ converges to some $b_t\in\R$.  $b_t$ is in fact positive since from the equation we deduce that the sequence $b_j$ is uniformly bounded from below by a positive quantity. The pair $ (\phi_t,b_t) $ solves $(*_t)$ and the closedness of $S$ follows.  
\end{proof}

We are ready to prove Theorem \ref{main}.

\begin{proof}[Proof of Theorem $\ref{main}$] 
In view of Lemma \ref{lemma3} for every $\phi \in C^{\infty}_B(M)$ we have 
$$
\frac{(\Omega+\partial \partial_J\varphi)^n}{\Omega^n}=
\Delta\varphi
+Q(\nabla \varphi,\nabla \varphi)+1\,,
$$
where $\Delta$ is the Riemannian Laplacian of $g$ 
and $Q\in \Gamma(T^*M\otimes T^*M)$ is negative semi-definite. 

Let $F\in C^{\infty}_B(M)$. Proposition \ref{prop} 
implies that equation
\begin{equation*}
\Delta\varphi +Q(\nabla \varphi,\nabla \varphi)+1=b\,{\rm e}^F\,,\quad \int_M\varphi {\rm Vol}_g = 0
\end{equation*}
has a solution $ (\phi,b) \in C^\infty(M)\times\R_+$. We observe that since $F$ is basic, then $\phi$ is necessarily basic too. Indeed by setting  
\[
\Psi(\psi)=\Delta\psi +Q(\nabla \psi,\nabla \psi)+1
\]
we have that for every $ X\in \Gamma(\mathcal{F}) $ condition $X(F)=0$ implies 
\[
0=X(\Psi(\phi))=\Psi_{*|\phi} (X(\phi))
\]
and since $ \Psi_{*| \phi} $ is a linear elliptic operator without free term, by the maximum principle $ X(\phi) $ must be constant and then necessarily zero. Hence $(\varphi,b)$ solves the quaternionic Monge-Amp\`ere equation \eqref{qMAe} and the claim follows. 
%\medskip 
%\medskip 
%
%
%
%Let $ Q\in \Gamma(T^*M\otimes T^*M) $ be negative semi-definite, $ F\in C^\infty_B(M) $ and consider equation \eqref{eqngen}, that is
%\begin{equation*}
%\Delta\varphi +Q(\nabla \varphi,\nabla \varphi)+1=b\,{\rm e}^F\,,\quad \int_M\varphi {\rm Vol}_g = 0\,.
%\end{equation*}
%By Proposition \ref{prop} there exists a solution $ (\phi,b) \in C^\infty(M)\times\R_+ $. We need to show that $ \phi $ is necessarily a basic function, whenever the datum $ F $ is. Consider the operator
%\[
%\Psi(\psi)=\Delta\psi +Q(\nabla \psi,\nabla \psi)+1\,.
%\]
%Since $ F $ is basic also $ \Psi(\phi) $ must be basic, therefore, for every $ X\in \Gamma(\mathcal{F}) $
%\[
%0=X(\Psi(\phi))=\Psi_{*|\phi} (X(\phi))
%\]
%and since $ \Psi_{*| \phi} $ is a linear elliptic operator without free term, by the maximum principle $ X(\phi) $ must be constant and then necessarily zero. It follows that $ \phi $ lies in $ C^\infty_B(M) $ and for basic functions, because of Lemma \ref{lemma3}, the quaternionic Monge-Amp\`ere equation \eqref{qMAeb} is precisely of the form \eqref{eqngen} for some $ Q $. The existence of a solution to \eqref{qMAeb} follows. Uniqueness has been shown in general at the beginning of section \ref{section}. 
\end{proof}

\section{An explicit example: ${\rm SU}(3)$}
In this section we observe that Theorem \ref{main} can be applied for instance to study the quaternionic Monge-Amp\`ere equation on  ${\rm SU}(3)$ endowed with Joyce's hypercomplex structure, which we are about to describe. 

\medskip
The Lie algebra of $ {\rm SU}(3)$ 
$$
\mathfrak{su}(3)=\left\{\begin{pmatrix}
D & f\\
-\bar{f}^t & -\mathrm{tr}(D)
\end{pmatrix}\,\,:\,\,\mbox{$ D\in \mathfrak{u}(2) $ and $ f\in \C^2 $}\right\}
$$
splits in 
\[
\mathfrak{su}(3)=\mathfrak{b} \oplus \mathfrak{d}\oplus \mathfrak{f}
\]
where
\begin{itemize}
\item $ \mathfrak{d}\cong \mathfrak{sp}(1) $ is the space of matrices with zero $ f $ and $ \mathrm{tr}(D) $;

\vspace{0.1cm}
\item $ \mathfrak{f} $ consists of matrices with zero $ D $;

\vspace{0.1cm}
\item $ \mathfrak{b}\cong \R $ is the set of diagonal matrices commuting with $ \mathfrak{d} $.
\end{itemize} 

We have
$$
[\mathfrak{b},\mathfrak{d}]=0\,,\quad [\mathfrak{b},\mathfrak{f}]=\mathfrak{f}\,,\quad 
 [\mathfrak{d},\mathfrak{f}]= \mathfrak{f}\,,\quad [\mathfrak{f},\mathfrak{f}]=\mathfrak{b}\oplus \mathfrak{d}\,,\quad 
[\mathfrak{d},\mathfrak{d}]=\mathfrak{d}\,.
$$
In particular $\mathfrak{b}\oplus \mathfrak{d}$ is a subalgebra of $\mathfrak{su}(3)$ and induces a foliation $\mathcal F$ on ${\rm SU}(3)$. The HKT structure of ${\rm SU}(3)$ can be described in terms of the \lq\lq standard\rq\rq\, basis of $\mathfrak{su}(3)$
\[
\{X_1,\dots,X_8\}
\]
given by 
\begin{align*}
X_1=\begin{pmatrix}
i & 0 & 0\\
0 & i & 0\\
0 & 0 & -2i
\end{pmatrix}, && X_2=\begin{pmatrix}
i & 0 & 0\\
0 & -i & 0\\
0 & 0 & 0
\end{pmatrix}, && X_3=\begin{pmatrix}
0 & 1 & 0\\
-1 & 0 & 0\\
0 & 0 & 0
\end{pmatrix}, && X_4=\begin{pmatrix}
0 & i & 0\\
i & 0 & 0\\
0 & 0 & 0
\end{pmatrix}, \\
X_5=\begin{pmatrix}
0 & 0 & 1\\
0 & 0 & 0\\
-1 & 0 & 0
\end{pmatrix}, && X_6=\begin{pmatrix}
0 & 0 & i\\
0 & 0 & 0\\
i & 0 & 0
\end{pmatrix}, && X_7=\begin{pmatrix}
0 & 0 & 0\\
0 & 0 & 1\\
0 & -1 & 0
\end{pmatrix}, && X_8=\begin{pmatrix}
0 & 0 & 0\\
0 & 0 & -i\\
0 & -i & 0
\end{pmatrix}.
\end{align*}
Following Joyce's paper \cite{J},
the hypercomplex structure on ${\rm SU}(3)$ is then defined by the following relations:  
\begin{itemize}
\item on $ \mathfrak{b}\oplus \mathfrak{d}=\braket{X_1,X_2,X_3,X_4} $ as $ IX_1=X_2 $, $ IX_3=X_4 $, $ JX_1=X_3 $, $ JX_2=-X_4 $;

\vspace{0.1cm}
\item on $ \mathfrak{f}=\braket{X_5,X_6,X_7,X_8} $ as $ Iv=[X_2,v] $, $ Jv=[X_3,v] $, $ Kv=[X_4,v] $ for every $ v\in \mathfrak{f} $.
\end{itemize}
According to \cite{Dimitrov-Tsanov,Gori} all invariant hypercomplex structures on compact Lie groups are obtained from Joyce's construction. In view of these results, up to equivalence, this is the unique hypercomplex structure on ${\rm SU}(3)$. Moreover the standard metric $g$ that makes $\{X_1,\dots,X_8\}$ into an orthogonal frame is HKT with respect to $(I,J,K)$.  
Hence the foliation $\mathcal F$ induced by $\mathfrak{b}\oplus \mathfrak{d}$ is $(I,J,K)$-invariant and all the assumptions of Theorem \ref{main} are satisfied. Therefore the quaternionic Calabi-Yau equation on $({\rm SU(3)},I,J,K,g)$
can be solved for every  $\mathcal F$-basic datum $F$.

\medskip
As an example we show explicitly how the proof of Lemma \ref{lemma3} works in the case of ${\rm SU}(3)$. Let $\{X^1,\dots,X^8\}$ be the dual basis of $\{X_1,\dots X_8\}$ and let 
$$ 
Z^1=-\frac{1}{2}(X^1+iX^2)\,,\quad  Z^2=\frac{1}{2}(X^3+iX^4)\,,\quad Z^3=\frac{1}{2}(X^5+iX^6)\,,\quad Z^4=\frac{1}{2}(X^7+iX^8)
$$
be the induce unitary coframe with respect to $(g,I)$. 
Since the only non-zero brackets of vectors in $\{X_1,\dots,X_8\}$ are   
\[
\begin{split}
&[X_5,X_6]=X_1+X_2\,, \qquad [X_7,X_8]=X_2-X_1\,, \qquad [X_3,X_4]=2X_2\,,\\
&\frac{1}{2}[X_2,X_4]=[X_5,X_7]=-[X_6,X_8]=-X_3\,,\\
&\frac{1}{2}[X_2,X_3]=-[X_5,X_8]=-[X_6,X_7]=X_4\,,\\
&\frac{1}{3}[X_1,X_6]=[X_2,X_6]=-[X_3,X_7]=-[X_4,X_8]=-X_5\,,\\
&\frac{1}{3}[X_1,X_5]=[X_2,X_5]=-[X_3,X_8]=[X_4,X_7]=X_6\,,\\
&\frac{1}{3}[X_1,X_8]=-[X_2,X_8]=-[X_3,X_5]=-[X_4,X_6]=X_7\,,\\
&\frac{1}{3}[X_1,X_7]=-[X_2,X_7]=-[X_3,X_6]=[X_4,X_5]=-X_8\,,
\end{split}
\]
we have 
\begin{align*}
\partial Z^1=0\,, && \partial Z^2=2Z^{12}+2Z^{34}\,, && \partial Z^3=(1+3i)Z^{13}\,, && \partial Z^4=(1-3i)Z^{14}\,. 
\end{align*}
For a basic function $\varphi$ we have $ Z_1(\phi)=Z_2(\phi)=0 $, where $ \{Z_1,\dots,Z_4\} $ is the dual of $ \{ Z^1,\dots,Z^4 \} $, and thus we obtain
\[
\begin{split}
\partial \partial_J\varphi=&-\partial J \left( Z_{\bar 3}(\varphi)\, Z^{\bar{3}}+Z_{\bar 4}(\varphi)\, Z^{\bar{4}} \right)=\partial \left( Z_{\bar 3}(\varphi)\, Z^4-Z_{\bar 4}(\varphi)\, Z^3 \right)\\
=&\left( Z_3Z_{\bar 3}(\varphi) +Z_4Z_{\bar{4}}(\phi) \right) Z^{34}-\left(Z_1Z_{\bar{4}}(\phi)+ 
(1+3i)Z_{\bar{4}}(\phi)\right)Z^{13}\\
&+\left(Z_{1}Z_{\bar{3}}(\phi) +(1-3i)Z_{\bar 3}(\varphi)\right)Z^{14}+Z_2Z_{\bar{3}}(\phi)\,Z^{24}-Z_2Z_{\bar{4}}(\phi)Z^{23}
\end{split}
\]
which, using
\[
[Z_1,Z_{\bar{4}}]=(1-3i)Z_{\bar{4}}\,, \qquad [Z_1,Z_{\bar{3}}]=(1+3i)Z_{\bar{3}}\,, \qquad [Z_2,Z_{\bar{3}}]=-2Z_4\,, \qquad [Z_2,Z_{\bar{4}}]=2Z_3\,,
\]
simplifies to
\[
\partial \partial_J\varphi=\left(Z_3Z_{\bar 3}(\varphi)+Z_4Z_{\bar{4}}(\phi) \right) Z^{34}-2Z_{\bar{4}}(\phi)\,Z^{13}+2Z_{\bar 3}(\varphi)\,Z^{14} -2Z_4(\phi)\,Z^{24}-2Z_{3}(\phi)Z^{23}\,.
\]
Taking into account that the HKT form is
\[
\Omega=Z^{12}+Z^{34}
\]
we obtain 
$$
\frac{(\Omega+\partial\partial_J\varphi)^2}{\Omega^2}=
1+Z_3Z_{\bar{3}}(\varphi)+Z_4Z_{\bar{4}}(\varphi)-4\abs{Z_3(\phi)}^2-4\abs{Z_4(\phi)}^2
$$
in accordance with Lemma \ref{lemma3}.

%To prove existence we approach the problem with the \emph{method of continuity}, for which uniform estimates on the solutions are essential. To prove the required estimates we apply results from the book of O. Ladyzhenskaya and N. Ural'tseva \cite{LU}, which treats the more general class of second order quasilinear elliptic equations.

\end{document}